\documentclass[12pt,a4paper]{amsart}
\usepackage{graphicx} % Required for inserting images

\usepackage{amsmath,amsthm,amscd,amssymb,eucal,mathrsfs}
\usepackage{amsfonts}
\usepackage{latexsym}
\usepackage{graphicx} 
\usepackage{color}
\usepackage{multicol,multirow}
\usepackage{dsfont}
\usepackage[all]{xy}
\usepackage{tcolorbox}

\newcommand{\mathset}[1]{{\left\{#1\right\}}}
\newcommand{\absolute}[1]{\left\lvert#1\right\rvert}
\newcommand{\norm}[1]{\left\|#1\right\|}

\newtheorem{thm}{Theorem}[section]
\newtheorem{prop}[thm]{Proposition}

\newtheorem{remark}[thm]{Remark}
\newtheorem{definition}[thm]{Definition}

\newtheorem{Thm}{Theorem}

\newtheorem*{Satz*}{Satz}
\newtheorem{task}{Task}

\newtheorem{Lemma}[thm]{Lemma}

\newtheorem{Corollary}[thm]{Corollary}

\newtheorem{Assumption}{Assumption}

\DeclareMathOperator{\Spec}{Spec}

\DeclareMathOperator{\Pic}{Pic}

\DeclareMathOperator{\Div}{Div}

\DeclareMathOperator{\divisor}{div}
\DeclareMathOperator{\ddc}{dd}

\DeclareMathOperator{\rank}{Ran}
\DeclareMathOperator{\average}{av}

\title[Theta-Induced Diffusion on Tate Curves]{Theta-Induced Diffusion  on Tate Elliptic Curves over Non-Archimedean Local Fields}

\author{Patrick Erik Bradley}

\date{\today}

\begin{document}

\maketitle

\begin{abstract}
A diffusion operator on the $K$-rational points of a Tate elliptic curve $E_q$ is constructed, where $K$ is a non-archimedean local field,  as well as an operator on the Berkovich-analytification $E_q^{an}$ of $E_q$. These are integral operators for measures coming from a regular $1$-form, and  kernel functions constructed via theta functions. The second operator can be described via certain  non-archimedan curvature forms on $E_q^{an}$.
The spectra of these  self-adjoint bounded operators on the Hilbert spaces of $L^2$-functions  are
identical and found to consist of finitely many eigenvalues. A study of the corresponding heat equations  yields a positive answer to the Cauchy problem, and  induced Markov processes on the curve.
Finally, some geometric information about the $K$-rational points of $E_q$ is retrieved from the spectrum.
\end{abstract}

%%%%%%%%%%%%%%%%%%%%

%%%%%%%%%%%%%%%%%%%%
\section{Introduction}

The uniformisation of a projective algebraic curve $X$ addresses the problem of finding its universal cover $\Omega$ and the action of its fundamental group $\Gamma$ such that $X\cong \Omega/\Gamma$. In the case of curves defined over the complex numbers, this
problem is solved by showing that the only simply connected Riemann surfaces are up to conformal equivalence the open unit disc, the complex plane and the Riemann sphere. Hence, in the case of an elliptic curve $E$, the fundamental group $\pi_1(E,0)$ can be  represented
as a lattice $\Lambda$ acting on $\mathds{C}$. In other words,
$E\cong \mathds{C}/\Lambda$. In the  case of a non-archimedean local field $K$, there are much more simply connected subdomains of $\mathds{P}_1(K)$. However, the Tate elliptic curve $E_q(K)$ does have a uniformisation of the form $E_q(K)\cong K^\times/q^{\mathds{Z}}
$ where the multiplicative group $K^\times$ is indeed simply connected, and the fundamental group is a multiplicative lattice depending on a parameter $q$ in the ring of integers of $K$.
Unlike in the complex case, not every elliptic curve has a uniformisation of this kind. It turns out that the ones for which this is possible are precisely the ones with split multiplicative  reduction \cite[Thm.\ V.5.3]{Silverman1994}.
\newline

Brownian motion can be modelled by the heat equation, which describes a diffusion process on a given space. One of its mathematical meanings is to be a tool for extracting information about the space from the diffusion equation. Already the spectrum of the corresponding Laplacian operator reveals something about the space. In the $p$-adic context, the heat equation has been studied only on local fields, cf.\ e.g.\ \cite{Zuniga2015}, and on subspaces of such, cf.\ e.g.\ \cite{ZunigaNetworks}, cf.\ also
 Brownian motion on $\mathds{Q}_p^d$ \cite{RW2023}.
 A multi-variate version of the $p$-adic heat equation is studied in \cite{RZ2008}, where it is shown that its fundamental solution is the transition density of a Markov process. 
The author also studied heat equations 
on Mumford curves \cite{brad_heatMumf}, where the approach was inspired by W.\ Z\'{u}\~{n}iga-Galindo's approach to studying $p$-adic heat equations on finite graphs \cite{ZunigaNetworks} and interpreting diffusion as local transitioning between the vertices of  a reduction graph of the curve along edges.
This means that only information about the skeleton of the  Mumford curve, and these include Tate curves, should be
retrievable through this approach.
A retrieval of a reduction graph in the case of a Mumford curve or a Mumford-uniformisable abelian variety using Vladimirov-Taibleson-type operators as studied 
in \cite{Taibleson1975,VVZ1994} was effected in \cite{BL_shapes_p} from their infinite spectra.
So, a diffusion operator allowing to extract geometric information about $p$-adic spaces becomes a desired object.
\newline

This present article addresses the mentioned desideratum in the simplest higher genus case of a Tate elliptic curve $E_q$ by using the invariant measure $\absolute{\omega}$ on $E_q(K)$ coming from a regular differential form $\omega$ on $E_q(K)$. Notice that $p$-adic elliptic curves which are also $1$-dimensional $p$-adic manifolds are necessarily Tate curves (possibly after a finite field extension), because elliptic curves with good reduction are not locally embeddable into the local field. So, this article studies the simplest example of a compact $p$-adic manifold which is also a projective algebraic variety.
\newline

The Berkovich analytification $E_q^{an}$ captures all  reduction graphs associated with the rational points of the curve for any (complete) extension field of $K$, it is natural to also construct an operator which acts on functions on $E_q^{an}$. 
In this way, the necessity to fix the field $K$ becomes obsolete, as long as it is sufficiently large.
This more unifying approach becomes possible through an integral operator which can be described as integrating over a signed Radon measure presented as the difference between two curvature forms of the type
\[
c_1\left(
%\mathcal{O}_{E_q^{an}},
\mathcal{L}
\norm{\cdot}\right)
\]
where $\norm{\cdot}$ is a continuous subharmonic metric on a  line bundle %$\mathcal{O}_{E_q^{an}}$
$\mathcal{L}$ which is ample. This so-called \emph{Chambert-Loir measure} was developped in the context of non-archimedean Arakelov theory and is in the present case a finite linear combination of Dirac measures on the skeleton of $E_q^{an}$, and the metrics are obtained because the non-archimedean Calabi-Yau problem for smooth strictly $K$-analytic curves has been solved in Thuillier's dissertation \cite{Thuillier2005}. This is a case in which the Monge-Amp\`ere equation
\[
c_1(\mathcal{L},\norm{\cdot})^n=\nu
\]
for a given positive Radon measure on the analytification of an $n$-dimensional algebraic variety $X$ and an invertible sheaf $\mathcal{L}$ on $X$ can be solved. In higher dimensions, this has also been solved under certain mild restrictions quite recently in e.g.\ \cite{BFJ2016,FGK2022}.
The diffusion operator on $E_q^{an}$ is obtained from the operator $\mathcal{H}_\theta$ on $E_q(K)$ described below essentially by pushing forward to the skeleton, and thus has an identical behaviour as $\mathcal{H}_\theta$ concerning the spectrum and the heat equation. 
\newline

Before describing the operator $\mathcal{H}_\theta$, let us view its first application (Theorem \ref{applicationTheorem}). It turns out that its spectrum can detect the presence or absence of two kinds of points:
\begin{itemize}
\item $2$-torsion points in $E_q(K)$,
\item third roots of certain types of points in $E_q(K)$ depending on the valuation $v(q)$ of the curve parameter $q$.
\end{itemize}
In the case of their presence, also the
absolute values of their representatives in the fundamental domain are detected.
Furthermore:
\begin{itemize}
\item the parity of the integer number $v(q)$,
\end{itemize} 
where $v(x)=-\log_{p^{f}}(x)$, is detectable from the spectrum. Here, $f$ is the degree of the residue field extension.
These results depend on knowing the field $K$, in particular its uniformiser $\pi$, and stands in contrast to \cite{BL_shapes_p}, where the reduction graph structure is revealed from the spectrum. This result opens a new bridge  between ultrametric diffusion and arithmetic. In  particular the fact that the detectability of whether certain kinds of $K$-rational points exist can now be read off the spectrum of a $p$-adic operator begs the question of a deeper connection between $p$-adic diffusion on an elliptic curve and its arithmetic.
\newline

The new diffusion operator developed here is, like all $p$-adic diffusion operators known so far, an integral operator whose kernel function is defined via the absolute value of a meromorphic function on the Tate curve. This generalises the usual $p$-adic kernel functions, since those are given by functions depending on the absolute value of a coordinate function on the affine line, aka radial functions.
However, unlike the Vladimirov-Taibleson operator, it is not expressed as a pseudodifferential operator, meaning that the  Fourier transform is not used here.
Its kernel function is constructed via theta functions on the multiplicative group $K^\times$ which are invariant under the uniformising group $q^{\mathds{Z}}$.
It turns out that, in this way, a bounded linear operator is obtained which is self-adjoint on the Hilbert space of $L^2$-functions with respect to integrating along the invariant measure $\absolute{\omega}$. More precisely:
\begin{Thm}\label{firstTheorem}
The space $L^2(E_q(K),\absolute{\omega})$ has an orthogonal decomposition
\[
L^2(E_q(K),\absolute{\omega})=L^2(E_q(K))_\sigma\oplus L^2(E_q(K))_0
\]
into $\mathcal{H}_{\theta}$-invariant subspaces. The subspace $L^2(E_q(K))_\sigma$ is of finite dimension
$v(q)$ and spanned by the indicator functions 
supported on the circles $S_\ell(0)$ of radius $p^{-f\ell}$ centred in zero, with $\ell = 0,\dots,v(q) - 1$,
and the subspace $L^2(E_q(K))_0$ is spanned by the wavelets on $E_q(K)$ obtained by pull-back of Kozyrev wavelets. The
spectrum of $\mathcal{H}_\theta$ as a linear operator on $L^2(E_q(K),\absolute{\omega})$ consists entirely of
eigenvalues. These are the eigenvalues of a certain matrix acting
on $L^2(E_q(K))_\sigma$, and the negative degree values  corresponding to
normalised wavelets on $E_q(K)$ supported in circles $S_k(x)$ with $x$ inside the annulus 
$\mathset{u\in K^\times\mid\absolute{q}<\absolute{u}\le 1}$, and 
$k = 0,\dots, v(q) - 1$.
\end{Thm}

This unsurprising result actually comes from viewing the Hilbert space $L^2(E_q(K),\absolute{\omega})$ as a decomposition into a direct sum of $L^2$-spaces of finitely many $p$-adic discs, each endowed with an individually scaled Haar measure coming from the differential form $\omega$. What is new is that here the degree eigenvalues are grouped according to the circles contained inside the annulus forming a fundamental domain for $E_q(K)$. Hence, there is not only an infinite part of their ``multiplicities'' coming from shrinking the wavelets, but also an inherent finite part coming from a partitioning of circles into finitely many maximal subdiscs.
\newline

The next result is about very general $p$-adic operators on $K$-analytic manifolds of any finite dimension giving rise to Feller semigroups and solutions to corresponding heat equations under some technical assumptions.
And the operators are more general of the form
\begin{align*}
%\label{advection}
\mathcal{J}f(x)=\int_X
\left\{j(x,y)f(y)-j(y,x)f(x)\right\}\absolute{\omega(y)}
\end{align*}
as considered in \cite{ZunigaEigen,ZunigaRugged}, where the kernel function $j(x,y)$ is such that the positive maximum principle is satisfied.

\begin{Thm}\label{secondTheorem}
(Under some technical assumptions) There exists a probability measure $p_t(x,\cdot)$ with $t\ge0$, $x\in X$, on the Borel $\sigma$-algebra of $X$ such that the Cauchy problem 
\begin{align*}
u(\cdot,t)&\in C^1\left([0,\tau],C_0(X,\mathds{R})\right)
\\
\frac{\partial}{\partial t}u(x,t)&
=\int_X \left(j(x,y)u(y,t)-j(y,x)u(x,t)\right)\absolute{\omega(t)},&t\in[0,\tau],\;x\in X
\\
u(x,0)&=u_0(x)\in C_0(X,\mathds{R})
\end{align*}
has a unique solution of the form
\[
h(t,x)=\int_{X}h_0(y)\,p_t(x,\absolute{\omega(y)})
\]
In addition, $p_t(x,\cdot)$ is the transition function of a strong Markov process whose paths are right continuous and have no discontinuities other than jumps.
\end{Thm}

%\begin{Thm}\label{secondTheorem}
%There exists a probability measure $p_t(x, \cdot)$ with 
%$t \ge 0$, $x \in E_q(K))$, on the Borel $\sigma$-algebra of $E_q(K)$ such that the Cauchy problem for the heat equation, with operator $\epsilon\mathcal{H}_\theta$ with $\epsilon>0$, has a unique solution of the form
%\[
%h(t, x) = \int_{E_q(K)}
%h_0 (y) p_t (x,\absolute{\omega})
%\]
%In addition, $p_t(x,\cdot)$ is the transition function of a strong Markov process whose paths
%are right continuous and have no discontinuities other than jumps.
%\end{Thm}

As a corollary, it then follows  that the new operator on $E_q(K)$ also  describes a strong Markov process on $E_q(K)$ whose paths are right continuous and have no discontinuities other than jumps. The Cauchy problem for the corresponding heat equation has a positive answer depending uniquely on the initial condition in $C(E_q(K))$. The spectrum of the operator acting on $L^2(K)$ consists of finitely many eigenvalues, a part of which correspond to eigenfunctions of the Laplacian of a complete finite graph, and the other part comes from an infinite family of functions which restrict to the well-known $p$-adic Kozyrev wavelets supported inside a fundamental domain of $E_q(K)$, which is an annulus in $K$. Because of the transitions on a complete graph, one can view this new operator as non-local on a Tate curve, which is in contrast to the ones from \cite{brad_heatMumf} on Mumford curves.
\newline

The general Theorem \ref{secondTheorem}, inspired by biology, cf.\ \cite{ZunigaEigen}, now begs the question for further research, whether or how it is possible to extract information about special $K$-rational points on elliptic curves via the study of a directed diffusion process.
\newline

Coming back to self-adjoint operators, it is shown  that  $\mathcal{H}_\theta$ can be extended to an operator on   $E_q^{an}$ with corresponding properties on $L^2(E_q^{an})$, cf.\ 
Corollary \ref{MarkovProcessBerkovichTateCurve}. And this extended operator $\mathcal{H}_{\theta,\sigma}$ can  be written as integration against the Chambert-Loir measure
$c_1(\mathcal{L},\norm{\cdot}_{\theta,x})$
for a suitably metrised ample line bundle $\mathcal{L}$:
\begin{Thm}\label{thirdTheorem}
The heat operator 
$\mathcal{H}_{\theta,\sigma}$ on $E_q^{an} (K)$ is obtained as an integral
operator of the form
\[
\mathcal{H}_{\theta,\sigma}\psi(x)=
\int_{E_q(K)^{an}}
\psi\, c_1\left(\mathcal{L},\norm{\cdot}_{\theta,x}\right)
\]
with metric
$\norm{\cdot}_{\theta,x}=e^{-g_{\theta,x}}$ and
\[
g_{\theta,x}=\sum\limits_{z\in\sigma(E_q^{an}(K))}
\alpha_zg_z
\]
where $\alpha_z>0$ and $g_z$ is a continuous subharmonic function on $E_q^{an}$ such that
\[
c_1\left(\mathcal{L},\norm{\cdot}_z\right)=\delta_z
\]
is a Dirac measure for $\norm{\cdot}_z=e^{-g_z}$, $y\in \sigma(E_q^{an}(K))\sqcup\sigma(x)$, and $x\in E_q^{an}(K)$.
\end{Thm}
Here, $\sigma$ is the retraction map of $E_q^{an}$ onto its skeleton.
\newline

A potential application outside of pure mathematics is seen in the analysis of topological data, in particular in the context of the ongoing DFG project \emph{Distributed Simulation of Processes in
Buildings and City Models}. In this context, an envisioned  heat flow approach could lead to the possibility of verifying the topological correctness of CAD models obtained from point clouds on distributed computing systems.
\newline

The following section introduces the invariant measure $\absolute{\omega}$, explains how to construct a kernel function from theta functions, and studies the spectrum of the diffusion operator.
The heat equation on  $E_q(K)$, and the development of a diffusion process on
$E_q^{an}$ and its relationship to Chambert-Loir measures
is treated in the third section. The short last section is devoted to extracting information about the Tate curve $E_q(K)$ from the spectrum.
\newline

Some words about notation are in order. The non-archimedean local field used  here is denoted as $K$. Its local ring is $O_K$ and has a unique maximal ideal $\mathfrak{m}_K$. The uniformiser of $K$ is denoted as $\pi$ and the absolute value $\absolute{\cdot}$ on $K$ is scaled such that
\[
\absolute{\pi}=p^{-f}
\]
where $f$ is the degree of the extension of the residue field $O_K/\mathfrak{m}_K$ over the finite field $\mathds{F}_p$ with $p$ elements. The prime $p$ is assumed not to be $2$ or $3$, as this is needed occasionally,  in order to 
not have to deal with the intricacies occurring with these primes. The Haar measure on $K$ is denoted as $\mu_K$, but also as $\absolute{dx}$, if $x$ is the variable of integration. The reason for writing it in this way is because we are also working with measures $\absolute{\omega}$ coming from a differential $1$-form, and these are locally written as 
\[
\omega|_U=f\,dx
\]
where $f$ is a function defined on an open piece $U$ of the space under consideration.
In this case, the measure is locally written as
\[
\absolute{\omega(x)}=\absolute{f(x)}\absolute{dx}
\]
if $x\in U$ is the variable of integration.
Indicator functions will be written as
\[
\Omega(x\in B)
\]
where $B$ is a measurable set,
or as 
\[
\Omega(\absolute{f(x)}):=
\Omega(x\in B(f))
\]
with
\[
B(f)=\mathset{x\in K\mid\absolute{f(x)}\le 1}
\]
for some function $f\colon K\to K$.
\newline

The following function spaces are used:
\begin{align*}
\mathcal{D}(X)&=\mathset{f\colon X\to\mathds{C}\mid \text{$f$ is locally constant with compact support}}
\\
C(X)&=\mathset{f\colon X\to\mathds{C}\mid\text{$f$ is continuous}}
\\
L^2(X,\nu)&=\mathset{f\colon X\to \mathds{C}\mid\text{$f$ is square integrable w.r.t.\ $\nu$}}
\end{align*}
where $\nu$ is a positive Borel measure on a topological space $X$.
The elements of $\mathcal{D}(X)$ are also called \emph{test functions} on $X$.
In the situation of this article, the space of test functions is dense in the other two. Here, the space $C(X)$ is a Banach space w.r.t.\ the supremum norm $\norm{\cdot}_\infty$, and the second space is a Hilbert space w.r.t.\ the pairing induced by $\nu$.
\newline

An introduction to the theory of Berkovich analytic spaces can be found e.g.\ in \cite{Temkin2015} or, of course, \cite{Berkovich1990}. The results of that theory needed here are the existence of a skeleton $I(X)$ of the Berkovich-analytification  $X^{an}$ of a smooth projective variety $X$ over $K$, and a map
\[
\sigma\colon X^{an}\to I(X)
\]
which is a deformation retraction. The topology of $X^{an}$ is Hausdorff, and $X^{an}$ is locally path-connected \cite{Berkovich1999}. 
In the $1$-dimensional case, the skeleton $I(X^{an})$ is a metrised graph which for the Tate curve can be identified with a circle. 

%%%%%%%%%%%%%%%%%%5
\section{Diffusion on the Rational Points of a Tate Elliptic Curve}

From
\cite{Tate1995,Silverman1994} and \cite[Ch.\ 5.1]{FP2004}, one can take a comprehensive view on Tate's elliptic curves. Invariant functions on Tate curves are constructed in \cite{Tate1995} and in \cite[Ch.\ 5.1]{FP2004} via Theta functions.
This allows to explicitly construct a measure and a kernel function.

\begin{thm}[Tate's Uniformisation of Tate Curves]\label{TateUniformisation}
Let $q\in K^\times$, $\absolute{q}<1$. Set
\[
s_k(q)=\sum\limits_{n\ge 1}\frac{n^kq^n}{1-q^n},\quad
a_4(q)=-5s_3(1),\quad
a_6(q)=\frac{-5s_3(q)+7s_5(q)}{12}
\]
Then $a_4(q)$ and $a_6(q)$ converge in $K$. The equation
\[
E_q\colon y^2+xy=x^3+a_4(q)+a_6(q)
\]
defines an elliptic curve with discriminant and $j$-invariant
\begin{align*}
\Delta(E_q)&=q\prod\limits_{n\ge1}(1-q^n)^{24}
\\
j(E_q)&=\frac{1}{q}+\sum\limits_{n\ge 0}c(n)q^n
\end{align*}
with $c(n)\in\mathds{Z}$.
\end{thm}

\begin{proof}
\cite[1, Thm.\ V.3.1]{Silverman1994}. Observe the typo in Silverman's book: the coefficient $5$ in $a_4(q)$ is missing there. Consult also \cite[\S3 (37)]{Roquette1970} or \cite[Thm.\ 5.1.10]{FP2004}.
\end{proof}

\begin{Lemma}\label{doubleAnnulus}
If
\[
\absolute{q}<\absolute{u}<\absolute{q}^{-1}
\]
then the coordinates of $E_q$ can be written as
\begin{align*}
X(u,q)&=\frac{u}{(1-u)^2}
+\sum\limits_{d\ge 1}\left(
\sum\limits_{m\mid d} m(u^m+u^{-m}-2)
\right)q^d
\\
Y(u,q)&=\frac{u^2}{(1-u)^3}
+\sum\limits_{d\ge 1}
\left(
\sum\limits_{m\mid d}
\left(
\frac{(m-1)m}{2}u^m
-\frac{m(m+1)}{2}u^{-m}+m
\right)
\right)q^2
\end{align*}
\end{Lemma}

\begin{proof}
\cite[Ch.\ V.3]{Silverman1994}.
\end{proof}

\begin{thm}[Tate]
The Tate elliptic curve $E_q$ can be written as
\[
E_q=\mathds{C}^\times_p/q^{\mathds{Z}}
\]
with $q\in K^\times$ such that $\absolute{q}<1$. The meromorphic functions on $E_q(K)$ are precisely the $q^{\mathds{Z}}$-invariant meromorphic functions on $K^\times$.
\end{thm}

\begin{proof}
\cite[Thm.\ 1]{Tate1995}, or \cite[Thm.\ 5.1.4]{FP2004} together with the Weierstrass model for $\mathds{C}_p/q^{\mathds{Z}}$ \cite[Thm.\ 5.1.10]{FP2004}.
\end{proof}

%%%%%%%%%%%%%%%%%%%%%%
\subsection{Measure on $E_q(K)$}

Any non-singular projective algebraic curve $X$ of positive genus has  regular differential $1$-forms. If $X$ is defined over a non-archimdean local field $K$, 
a regular differential $1$-form $\omega$ on $X$ gives rise to a positive measure $\absolute{\omega}$ on the space $X(K)$ of $K$-rational points outside the vanishing locus of $\omega$ \cite[Ch.\ 7.4]{IgusaLocalZeta}. A.\ Weil calls this a \emph{gauge measure} \cite[Ch.\ II.2.2]{WeilAAG}.  
Since the zero set of $\omega$ locally on each chart $U\to K$ has Haar measure zero \cite[Lem.\ 3.1]{Yasuda2017}, one has a natural extension of $\absolute{\omega}$ to all of $X(K)$, where it locally on a chart $U$ takes the form
\[
\int_U\absolute{\omega}=\int_{U'}\absolute{f_U}d\mu_{K}
\] 
with $\mu_K$ the Haar measure on $K$, and $\omega|_U=f_U\,dx$ for some $K$-analytic function  $f_U$ with local coordinate $x$.
Notice that in this description here, it is assumed that $X(K)$ is a $1$-dimensional $K$-analytic manifold, which is the case only if $X$ is a Mumford curve.
\newline

The Haar measure on $K$ will also be written as 
\[
\mu_K=\absolute{dx}
\]
in order to make visible the local coordinate $x$, and in order to distinguish it from the differential $1$-form $dx$ on $\mathds{A}^1_K$. In the case of the measure $\absolute{\omega}$, this will be effected by writing
\[
\absolute{\omega(x)}
\]
in particular when a clarification of the variable of integration is needed.
\newline

Returning to the Tate elliptic curve $E_q$, now assume that the parameter $u$ from Lemma \ref{doubleAnnulus} is in the annulus
\[
A_K(q)=\mathset{u\in K^\times\mid\absolute{q}<\absolute{u}\le 1}
\]
and look at the holomorphic differential
\begin{align}\label{invariantDifferential}
\omega=\frac{dx}{2y+x}
=\frac{dX(u,q)}{2Y(u,q)+X(u,q)}
\end{align}
on $E_q$. By the Riemann-Roch Theorem \cite[Prop.\ 5.1.2(2)]{FP2004}, any holomorphic differential $1$-form has no zeros on non-singular projective curves of genus $1$.
Hence, the measure $\absolute{\omega}$ 
corresponding to $\omega$, defined as in (\ref{invariantDifferential}),
is a gauge form in the sense of \cite[Ch.\ II.2.2]{WeilAAG}. The following Lemma has an explicit description of this measure:

\begin{Lemma}\label{HaarMeasureEllipticCurve}
On the annulus $A_K(q)$, the measure $\absolute{\omega}$ coming from $\omega$, defined as in (\ref{invariantDifferential}), takes the form
\[
\absolute{\omega(u)}=\frac{\absolute{du}}{\absolute{u}}
\]
for $u\in A_K(q)$,
where it
is a positive measure  which defines an $E_q(K)$-invariant measure on the $K$-rational points $E_q(K)$ of the Tate  curve $E_q$.
\end{Lemma}

\begin{proof}
With 
\[
F(u,q)=2Y(u,q)+X(u,q)
\]
one obtains
\[
\omega=\omega(u)=F(u,q)^{-1}\frac{dX(u,q)}{du}\,du
\]
This yields the real-valued measure
\[
\absolute{\omega(u)}
=\absolute{F(u,q)^{-1}\frac{dX(u,q)}{du}}\absolute{du}
\]
Now, Lemma \ref{doubleAnnulus} implies that
\begin{align*}
\absolute{F(u,q)^{-1}\frac{dX(u,q)}{du}}
&=\absolute{\frac{(1-u)^3}{u(1+u)}\cdot\frac{1+u}{(1-u)^3}}+\text{higher order terms in $q$}
\\
&=\frac{1}{\absolute{u}}
\end{align*}
because all coefficients of $q^d$ with $d\ge1$ each have absolute value at most one, and because
$\absolute{q}<\absolute{u}\le1$. It follows that
\[
\absolute{\omega(u)}=\frac{\absolute{du}}{\absolute{u}}
\]
as asserted. Since $A_K(q)$ is a fundamental domain of the action of the discrete group $q^{\mathds{Z}}$, this proves that $\absolute{\omega(u)}$ with $u\in A_K(q)$ defines the positive measure on $E_q$ associated with the holomorphic differential form $\omega$.

\smallskip
In order to show how the measure $\absolute{\omega}$ pulls back under the action of the group $E_q(K)$, denote
with $f_P\colon E_q\to E_q$ the translation with $P\in E_q(K)$. We claim that
\[
f_P^*\absolute{\omega}=\absolute{\omega}
\]
For this, w.l.o.g.\ take a representative $v$ of $P$ in the annulus $A_K(q)$, since $A(q)$ is a fundamental domain for $E_q(K)$. Then it holds true that
\[
f_P^*\absolute{\omega(u)}
=\absolute{\omega(vu)}
=\frac{\absolute{dvu}}{\absolute{vu}}
\stackrel{(*)}{=}
\frac{\absolute{v}\absolute{du}}{\absolute{v}\absolute{u}}
=\frac{\absolute{du}}{\absolute{u}}
=\absolute{\omega(u)}
\]
as asserted, whereby $(*)$ holds true because of the transformation rule for the Haar measure $\absolute{du}$ on $K$. 
\end{proof}

%%%%%%%%%%%%%%%%%%%
\subsection{A Kernel  Function on a Tate Curve}
A theta function on $K^\times$ for $E_q(K)$ is given by
\[
\theta(z)=
\prod\limits_{n\ge0}\left(1-q^nz^{-1}\right)\prod\limits_{n>0}\left(1-q^nz\right)
\]
with $z\in K^\times$, \cite[Def.\ 5.1.8]{FP2004}.
For $x\in K^\times$, one defines
\[
\theta_x(z)=\theta(x^{-1}z)
\]
and obtains that  any invertible meromorphic function on $E_q(K)$ can be written as
\begin{align}\label{meromorphicTheta}
f(z)=\lambda\prod\limits_{i=1}^N\theta_{x_i}(z)^{n_i}
\end{align}
where $\lambda\in K^\times$, and
$x_i\in K^\times$ are the representatives of the finitely many zeros and poles of $f$ with corresponding multiplicities $n_i\in\mathds{Z}$, provided
that
\begin{align}\label{divisor}
\sum\limits_{i=1}^Nn_i=0\quad\text{and}\quad \prod\limits_{i=1}^N x^{n_i}\in q^{\mathds{Z}}
\end{align}
cf.\ \cite[p.\ 128]{FP2004}. Notice that the first condition in (\ref{divisor}) is because the degree of $\divisor(f)$ is zero according to Riemann-Roch \cite[Prop.\ 5.1.2(2)]{FP2004}. 

\begin{remark}\label{pullbackHaarMeasure}
Notice that the annulus
$A_K(q)$
is a fundamental domain for the action of $q^{\mathds{Z}}$ on $K^\times$ giving rise to the universal covering map 
\[
\rho\colon K^\times \to E_q(K)
\]
and Lemma \ref{HaarMeasureEllipticCurve} says that the measure $\absolute{\omega}$ pulls back to the Haar measure on $K^\times$, normalised such that  circles centred in zero all have volume $1-p^{-f}$, and so is itself  a Haar measure on the Tate curve $E_q(K)$.
\end{remark}

Let $x,y\in E_q(K)$, and define the function
\begin{align}\label{preKernelFunction}
g(x,y)=\frac{\theta(x^{-1}y)\theta(y^{-1}x)}{\theta(xy)^2}
=\frac{\theta_x(y)}{\theta_{x^{-1}}(y)}\cdot
\frac{\theta_y(x)}{\theta_{y^{-1}}(x)}
\end{align}
which is $q^{\mathds{Z}}$-invariant in both variables. As each factor on the right-hand side of (\ref{preKernelFunction}) satisfies condition (\ref{divisor}), it defines a meromorphic function on $E_q(K)$ for any fixed $x\in E_q(K)$, and also for any fixed $y\in E_q(K)$.
Its divisor as a function on the surface $E_q(K)^2$ is
\[
\divisor(g)=2\left[y-x\right] - 2\left[y-x^{-1}\right]\in\Div\left(E_q(K)^2\right)
\]
whose double zero is the diagonal
$V(x-y)$, and whose double pole is the curve $V(xy-1)$ in $E_1(K)^2$, where
\[
V(F)=\mathset{x\in X\mid F(x)=0}
\]
is the vanishing set of a function $F\in K(X)$ for a variety $X$ over $K$.

\begin{Lemma}\label{absRationalFunction}
It holds true that
\[
\absolute{g(x,y)}
=\frac{\absolute{\tilde{x}\tilde{y}}\absolute{\tilde{x}-\tilde{y}}^2}{\absolute{1-\tilde{x}\tilde{y}}^2}
\]
for suitable representatives $(\tilde{x},\tilde{y})\in \left(K^\times\right)^2$ of $x,y\in E_q(K)^2$. 
\end{Lemma}

\begin{proof}
This follows from the fact that, if $z\in A_q(K)$, then
\[
\absolute{\theta(z)}
=\absolute{1-z^{-1}}
\]
which implies
\begin{align*}
\absolute{g(x,y)}
&=\frac{\absolute{1-\tilde{y}^{-1}\tilde{x}}\absolute{1-\tilde{x}^{-1}\tilde{y}}}{\absolute{1-\tilde{x}^{-1}\tilde{y}^{-1}}^2}
\\
&=\frac{\absolute{\tilde{x}-\tilde{y}}^2}{\absolute{\tilde{x}-\tilde{y}^{-1}}\absolute{\tilde{x}^{-1}-\tilde{y}}}
=\frac{\absolute{\tilde{x}\tilde{y}}\absolute{\tilde{x}-\tilde{y}}^2}{\absolute{1-\tilde{x}\tilde{y}}^2}
\end{align*}
as asserted.
\end{proof}

A neighbourhood of the poles of the map $A_K(q)^2\to\mathds{R},\;(x,y)\mapsto\absolute{g(x,y)}$ is the set
\[
\mathcal{P}_g=\mathset{(x,y)\in A_K(q)^2\mid\absolute{xy}=1\;\text{and}\;\absolute{x}\neq \absolute{y}}
\]
where $\absolute{x}\neq\absolute{y}$ is required, because if $\absolute{x}=\absolute{z}$ with
$z\in E_q(K)[2]\cap A_K(q)$, then
$x=uz$ with $\absolute{u}=1$ and
\[
\frac{\absolute{x-y}}{\absolute{1-xy}}
=\frac{\absolute{x}\absolute{1-x^{-1}y}}{\absolute{1-xy}}=\frac{\absolute{x}\absolute{1-uzy}}{\absolute{1-uzy}}
=\absolute{x}
\]
for any $y\in A_K(q)$.
This means that  pairs $(x,y)$ with $\absolute{x}=\absolute{y}$ are not poles of $\absolute{g(x,y)}$.
The $\absolute{\omega}$-measure of the set $\mathcal{P}_g$ is non-zero, because it is a union of circles. Hence, it is reasonable to define the following kernel function:
\[
H_{\theta}(x,y)=\begin{cases}
\absolute{g(x,y)},&\nexists\;\text{representatives}\;
(\tilde{x},\tilde{y})\in\mathcal{P}_g
\\
1,&\exists\;\text{representatives}\;(\tilde{x},\tilde{y})\in\mathcal{P}_g
\end{cases}
\]
for $x,y\in E_q(K)$.
In this way, obtain the linear operator
\[
\mathcal{H}_{\theta}
\psi(x)
=\int_{E_q(K)}H_{\theta}(x,y)(\psi(y)-\psi(x))\absolute{\omega(y)}
\]
for $\psi\in\mathcal{D}(E_q(K))$ and its extension to a linear operator on $C(E_q(K))$ and on $L^2(E_q(K))$. 
The significance of this kernel function is that it is controlled from the universal covering of the Tate curve via the function $\theta$. This is an example of an operator invariant under the action of the fundamental group of $E_q(K)$, and generalises the concept of radial functions which is used for constructing kernel functions for integral operators on function spaces over $K$.
\newline

The corresponding degree function is
\[
\deg_{\mathcal{H}_{\theta}}\colon E_q(K)\to\mathds{R},\;
x\mapsto
\int_{E_q(K)}H_{\theta}(x,y)\,\absolute{d\omega(y)}
\]
and is defined wherever the integral converges. The adjacency operator is defined as
\[
\mathcal{A}_{\theta}
\psi(x)=\int_{E_q(K)}
H_{\theta}(x,y)\psi(y)\absolute{\omega(y)}
\]
for $x\in E_q(K)$.
\newline

Given a partition of unity
\[
1=\sum\limits_{k=0}^{v(q)-1}\eta_\ell
\]
with
\[
\eta_\ell(x)
=\Omega(x\in S_\ell)
\in\mathcal{D}(E_q(K))
\]
where $S_\ell$ is the circle
\[
S_\ell=S_\ell(0)
=\mathset{x\in K\mid\absolute{x}=\absolute{\pi^\ell}}
\]
for $\ell\in\mathds{Z}$.
Of interest are the matrix elements
\[
A_{\sigma}\left(k,\ell\right)
=\frac{\mathcal{A}_{\theta}\,\eta_\ell(\pi^k)}{1-\absolute{\pi}}
\]
for $k,\ell\in\mathds{Z}/v(q)\mathds{Z}$.
\newline

In the following, assume that $\absolute{q}<\absolute{\pi}^2$.
Furthermore, all exponents $m$ in expressions like
\[
\absolute{w}^{m}
\]
for $w\in A_K(q)$ representing a point in $E_q(K)$
are understood modulo $v(q)$, if not stated otherwise.

\begin{Lemma}\label{matrixElements}
It holds true that
\[
A_\sigma(k,\ell)
=\begin{cases}
\absolute{\pi}^{k+\ell+2\min(k,\ell)},&k+\ell\not\equiv 0,\;k\not\equiv\ell\mod v(q)
\\
\frac{\absolute{\pi}^{4\ell}}{1-\absolute{\pi}^3},
&k+\ell\not\equiv 0,\;k\equiv\ell\mod v(q) 
\\
1,&k+\ell\equiv 0\mod v(q)
\\
\end{cases}
\]
where in  the expression $\min(k,\ell)$, it is assumed that 
$k,\ell\in\mathset{0,\dots,v(q)}$.
%$\mathds{Z}/v(q)\mathds{Z}$. 
\end{Lemma}

\begin{proof}
\emph{Case $k+\ell\equiv 0\mod v(q)$.}
This happens iff 
$\absolute{\pi^k\pi^\ell}\in\absolute{q}^{\mathds{Z}}$. In this case, 
\begin{align*}
\mathcal{A}_{\theta}\eta_\ell(\pi^k)&=
\int_{\absolute{y}=\absolute{\pi}^\ell}\absolute{\omega(y)}=1-\absolute{\pi}
\end{align*}
which proves the assertion in this case.

\smallskip\noindent
\emph{Case $k+\ell\not\equiv 0\mod v(q)$ and $k\not\equiv \ell\mod v(q)$}. 
This happens iff 
$\absolute{\pi^k\pi^\ell}\notin q^{\mathds{Z}}$ and $\absolute{\pi^k}\neq\absolute{\pi^\ell}$.
In this case, first assume that $0\le k<\ell<v(q)$. Then
\begin{align*}
\mathcal{A}_{\theta}\eta_\ell(\pi^k)&
=\absolute{\pi}^{k+\ell}\int_{\absolute{y}=\absolute{\pi}^\ell}
\absolute{\pi}^{2k}\frac{dy}{\absolute{y}}
\\
&=
\absolute{\pi}^{3k+\ell}(1-\absolute{\pi})
\end{align*}
Now, assume that $0\le\ell<k<v(q)$.
Then
\begin{align*}
\mathcal{A}_{\theta}\eta_\ell(\pi^k)&
=\absolute{\pi}^{k+\ell}
\int_{\absolute{y}=\absolute{\pi}^\ell}
\absolute{\pi}^{2\ell}
\frac{\absolute{dy}}{\absolute{y}}
\\
&=\absolute{\pi}^{k+3\ell}(1-\absolute{\pi})
\end{align*}
This proves the assertion in this case.

\smallskip\noindent
\emph{Case $k+\ell\not\equiv0\mod v(q)$ and $k\equiv\ell\mod v(q)$}. This is the case iff  $\absolute{\pi^k}=\absolute{\pi^\ell}$. Then
\begin{align*}
\mathcal{A}_{\theta}\eta_\ell(\pi^k)
&=\absolute{\pi}^{2\ell}
\int_{\absolute{y}=\absolute{\pi}^\ell}\absolute{\pi^\ell-y}^2\frac{\absolute{dy}}{\absolute{y}}
\\
&=\absolute{\pi}^\ell
(1-\absolute{\pi})
\sum\limits_{\nu=\ell}^\infty\absolute{\pi}^{3\nu}
\\
&=\frac{\absolute{\pi}^{4\ell}(1-\absolute{\pi})}{1-\absolute{\pi}^3}
\end{align*}
This proves the assertion in the remaining case.
\end{proof}

\begin{Corollary}\label{degreeFunction}
It holds true that
\[
\frac{\deg_{\mathcal{H}_{\theta}}(x)}{1-\absolute{\pi}}
=1+\absolute{\pi}^k
\sum\limits_{\ell=0\atop v(q)\nmid\ell+k}^{k-1}
\absolute{\pi}^{3\ell}
+\absolute{\pi}^{3k}
\sum\limits_{\ell=k+1\atop v(q)\nmid \ell+k}
^{v(q)-1}\absolute{\pi}^\ell
+\epsilon_q(k)\absolute{\pi}^{4k}
>0
\]
where $\absolute{x}=\absolute{\pi}^k$ and
\[
\epsilon_q(k)
=\begin{cases}
\left(1-\absolute{\pi}^3\right)^{-1}
,&v(q)\in 2\mathds{Z}\;\text{and}\;v(q)\nmid k
\\
1,&\text{otherwise}
\end{cases}
\]
with $k\in\mathset{0,\dots,v(q)-1}$.
\end{Corollary}

\begin{proof}
The expression is  immediately obtained from Lemma \ref{matrixElements}.
%Since the positive contributions are all at least the left hand side of the following
%inequality:
%\[
%1+\absolute{\pi}^4
%>\epsilon_q(k)\absolute{\pi}^4
%\]
%which holds true iff
%\[
%\absolute{\pi}^{-4}>
%\frac{2-\absolute{\pi}^3}{1-\absolute{\pi}^3}
%\]
%which is true.
\end{proof}

\begin{Lemma}
The following statements hold true:
\begin{enumerate}
\item $H_{\theta}\in\mathcal{D}(E_q(K)^2)$.
\item The degree function $\deg_{\mathcal{H}_{\theta}}$ is defined everywhere on $E_q(K)$.
\item If $\sigma(x)=\sigma(x')$, then $\deg_{\mathcal{H}_{\theta}}(x)=\deg_{\mathcal{H}_{\theta}}(x')$. In particular, the degree function is locally constant on $E_q(K)$.
\end{enumerate}
\end{Lemma}

\begin{proof}
1. Lemma \ref{absRationalFunction} shows that $H_{\theta}$ is locally constant on the compact space $E_q(K)^2$, i.e.\ $H_{\theta}\in\mathcal{D}(E_q(K)^2)$.
%\newline

\smallskip\noindent
2. This follows from 1.
%\newline

\smallskip\noindent
3. This is an immediate consequence of Corollary \ref{degreeFunction}.
\end{proof}

%%%%%%%%%%%%%%%%%%%%
\subsection{Spectrum of the Diffusion Operator on $E_q(K)$}

The pairing $\langle\cdot,\cdot\rangle_{\omega}$ on $L^2(E_q(K),\absolute{\omega})$ is defined as
\[
\langle\phi,\psi\rangle_{\omega}
=\int_{E_q(K)}\phi(x)\overline{\psi(x)}\absolute{\omega(x)}
\]

\begin{Lemma}
The operator $\mathcal{H}_{\theta}$ has the following properties:
\begin{enumerate}
\item $\mathcal{H}_{\theta}$ is a bounded linear operator on $C(E_q(K))$ w.r.t.\ $\norm{\cdot}_\infty$.
\item $\mathcal{H}_{\theta}$ is a self-adjoint bounded linear operator on $L^2(E_q(K),\absolute{\omega})$.
\end{enumerate}
\end{Lemma}

\begin{proof}
1. This is clear, since $H_{\theta}\in\mathcal{D}(E_q(K)^2)$ and $E_q(K)$ is a compact space.

\smallskip\noindent
2. Since the kernel function $H_{\theta}$ is symmetric, it follows that $\mathcal{H}_{\theta}$ is a symmetric operator on $L^2(E_q(K),\absolute{\omega})$ which is also everywhere defined. By the Hellinger-Toeplitz Theorem \cite[Thm.\ 2.10]{Teschl2010}, it follows that $\mathcal{H}_{\theta}$ is also bounded on $L^2(E_q(K),\absolute{\omega})$.
\end{proof}

The following matrix is helpful:
\begin{align}\label{helpfulMatrix}
L=(L_{k,\ell})\in\mathds{R}^{v(q)\times v(q)}
\end{align}
with
\[
L_{k,\ell}=\mathcal{H}_{\theta}\,\eta_\ell(\pi^k)
\]
\begin{Lemma}
The matrix $L$ is diagonalisable. Its non-zero eigenvalues are all negative real numbers. The eigenvectors corresponding to eigenvalue zero are the constant vectors. 
\end{Lemma}

\begin{proof}
All matrix elements 
\[
\mathcal{A}_{\theta}\eta_\ell(\pi^k)=(1-\absolute{\pi})A_\sigma(k,\ell)
\]
are all positive, cf.\ Lemma \ref{matrixElements}, and the degree function $\deg_{\mathcal{H}_{\theta}}$ is strictly positive, cf.\ Corollary \ref{degreeFunction}.
Also, the matrix $L$ is symmetric.
It follows that $L$ is the Laplacian of a positively weighted,  undirected complete graph $G$
with $v(q)$ nodes. Hence, $L$ is diagonalisable, and all eigenvalues are real, non-positive.
Since the graph $G$ is connected, it follows that the only eigenvectors corresponding to eigenvalue zero are the constant vectors.
\end{proof}

In \cite{Kozyrev2002}, S.V.\ Kozyrev introduced $p$-adic wavelets and found they provide for an orthogonal decomposition of $L^2(\mathds{Q}_p)$
into eigenspaces of the Vladimirov operator, meaning that its spectrum consists of eigenvalues corresponding to those wavelets, and that they are mutually orthogonal \cite[Thm.\ 2]{Kozyrev2002}. They are defined as
\[
\psi_{\gamma j n}(x)=p^{-\frac{\gamma}{2}}\chi(p^{\gamma-1}jx)\Omega(\absolute{p^\gamma x-n}_p)
\]
with $\gamma\in\mathds{Z}$, $n\in\mathds{Q}_p/\mathds{Z}_p$, and $j=1,\dots,p-1$, and where $\chi\colon K\to S^1$ is an additive  character into the complex unit circle.
For a local field $K$, this extends readily to functions
\[
\psi_{B,j}(x)=\mu_K(B)^{\frac12}
\chi_K\left(\pi^{d-1}\tau(j)\right)\Omega(x\in B)
\]
where $\tau\colon O_K/\mathfrak{m}_K\to O_K$ is a lift of the residue field of $K$, $j\in \left(O_K/\mathfrak{m}_K\right)^\times$, and $B\subset K$ a disc of radius $\absolute{\pi}^{-d}$. These are the \emph{Kozyrev wavelets} for $K$. It is well-known that the result \cite[Thm.\ 2]{Kozyrev2002} readily extends to the local field $K$. However, not all of this theory is needed, here. What is needed, is the following result:
\begin{Lemma}[Kozyrev]\label{KozyrevIntegral}
It holds true that
\[
\int_{W}\psi_{B,j}(y)\,\absolute{dy}=0
\]
for all measurable subsets $W\subseteq K$ containing $B$, and $j\in O_K/\mathfrak{m}_K$.
\end{Lemma}

\begin{proof}
\cite[Thm.\ 3.2.9]{XKZ2018}.
\end{proof}

\begin{definition}
A \emph{wavelet} on $E_q(K)$ is a function $\psi\colon E_q(K)\to\mathds{C}$ supported inside an open $U\subset E_q(K)$ such that
there is a local chart $\kappa\colon U\to K$ with $\kappa_*\psi$ a Kozyrev wavelet with support inside  $U$.
\end{definition}

It is immediately seen that a wavelet on $E_q(K)$ is defined by a Kozyrev wavelet $\psi_{B,j}$ supported inside $A_K(q)$. For simplicity, a wavelet on $E_q(K)$ will be written as such a Kozyrev wavelet $\psi_{B,j}$.

\begin{Lemma}
Let $\psi_{B,j}$ be a wavelet on $E_q(K)$. Then
\[
\int_W\psi_{B,j}\absolute{\omega}=0
\]
for all measurable subsets $W\subseteq E_q(K)$ containing the ball $B$, and $j\in O_K/\mathfrak{m}_K$.
\end{Lemma}

\begin{proof}
View $\psi_{B,j}$ as a Kozyrev wavelet supported in $A_K(q)$.
Since any ball in $A_K(q)$ is contained in some sphere
\[
S_k(x)=\mathset{z\in K\mid\absolute{z-x}=\absolute{\pi}^k}\subset A_K(q)
\]
with $k\in\mathset{0,\dots,v(q)-1}$, obtain 
\[
\int_{W}\psi_{B,j}\absolute{\omega}=\int_{W'}\psi_{B,j}(y)\frac{\absolute{dy}}{\absolute{y}}
=\absolute{x}^{-1}\int_{W'}\psi_{B,j}(y)\absolute{dy}=0
\]
where $W'\subset A_K(q)$ is a measurable set representing $W$, and the last equality follows from Lemma \ref{KozyrevIntegral}.
\end{proof}

\begin{thm}\label{SpectrumHeatOperator}
The space $L^2(E_q(K),\absolute{\omega})$ has an orthogonal decomposition
\[
L^2(E_q(K),\absolute{\omega})=L^2(E_q(K))_\sigma\oplus L^2(E_q(K))_0 
\]
into $\mathcal{H}_\theta$-invariant subspaces. The subspace
$L^2(E_q(K))_\sigma$ is of finite dimension $v(q)$ and spanned by the indicator functions $\eta_\ell$, with $\ell=0,\dots,v(q)-1$, and the subspace $L^2(E_q(K))_0$ is
spanned by the wavelets on $E_q(K)$.
The spectrum
of $\mathcal{H}_{\theta}$ as a linear operator on $L^2(E_q(K),\absolute{\omega})$ consists entirely of eigenvalues. These are the eigenvalues of the helpful matrix $L$ of (\ref{helpfulMatrix}) acting on $L^2(E_q(K))_\sigma$, 
and the negative degree values $-\deg_{\mathcal{H}_{\theta}}(x)$
corresponding to normalised wavelets on $E_q(K)$ supported in $S_k(x)$ with $x\in  A_K(q)$, and $k=0,\dots,v(q)-1$.
\end{thm}

This is Theorem \ref{firstTheorem} of the Introduction.

\begin{proof}
Since
\[
\langle \psi_{B,j},\eta_\ell\rangle_\omega=0
\]
for any pair $(B,j)$ consisting of a
ball $B\subset A_K(q)$ and $j\in O_K/\mathfrak{m}_K$ on the one hand, and of $\ell\in\mathset{0,\dots,v(q)-1}$, it follows immediately that the given decomposition is an orthogonal decomposition of $L^2(E_q(K))$.

\smallskip
Clearly, the space $L^2(E_q(K))_\sigma$ is of finite dimension $v(q)\in\mathds{N}$, spanned by the indicator functions $\eta_\ell$ with $\ell=0,\dots,v(q)-1$. 
Its orthogonal complement $L^2(E_q(K))_\sigma^\perp$ contains
the space $L^2(E_q(K))_0$ spanned by the wavelets, and these form an orthonormal set in $L^2(E_q(K))^\perp_\sigma$. In order to see the converse inclusion, observe that $L^2(E_q(K))_0$ is isomorphic to a direct sum of spaces $L^2(B_r(a))_0$, where the $B_r(a)$ are $p$-adic balls  forming a finite disjoint covering of $E_q(K)$, and
\[
L^2(B_r(a))=\mathds{C}\,\Omega(x\in  B_r(a))\oplus L^2(B_r(a))
\]
with $L^2(B_r(a))$ having an orthonormal basis consisting of Kozyrev wavelets supported in $B_r(a)$, cf.\ \cite[Prop.\ 2]{ZunigaEigen} (notice the missing subscript $_0$ for the space in the last equation of \emph{loc.\ cit.}). The only difference is that in the actual decomposition of $L^2(E_q(K),\absolute{\omega})$, each of the spaces $L^2(B_r(a))$ uses a rescaled Haar measure, where the rescaling comes from the invariant differential form $\omega$.
%let $\phi\in L^2(E_q(K))^\perp_\sigma$. It can be written as a function on $K$ supported in $A_K(q)$. Since 
%$\mathcal{D}(A_K(q))$ is dense in both Hilbert spaces $L^2(E_q(K),\absolute{\omega})$ and $L^2(A_K(q),\mu_K)$, one can now first assume  $\phi\in\mathcal{D}(A_K(q))$, and
%expand $\phi$ w.r.t.\ the Kozyrev wavelet basis of the Hilbert space $L^2(A_K(q),\mu_K)$, and thus automatically obtain a wavelet expansion of $\phi$ in $L^2(E_q(K),\absolute{\omega})$. 
%This proves that 
%\[
%L^2(E_q(K))_\sigma^\perp\cap\mathcal{D}(A_K(q))\subseteq L_2(E_q(K))_0
%\]
%By denseness, it now follows that 
%$L_2(E_q(K))_\sigma^\perp\subseteq L_2(E_q(K))_0$. 
This now proves the asserted 
orthonormal basis of $L^2(E_q(K),\absolute{\omega})$.

\smallskip
The eigenvalue of $\mathcal{H}_\theta$ corresponding to a wavelet $\psi_{B,j}$ with $B\subset S_k(x)\subset A_K(q)$ is now readily seen to be $-\deg_{\mathcal{H}_\theta}(x)$, as asserted. This proves the Theorem.
\end{proof}

\begin{remark}
Since the support of a wavelet may be arbitrarily small, it follows that the operator $\mathcal{H}_\theta$ is not a compact linear operator on $L^2(E_q(K))$.
\end{remark}

%%%%%%%%%%%%%%%%%%%55
\section{The heat equation}

Here, the heat equation
\begin{align}\label{heatEquation}
\left(\frac{\partial}{\partial t}-\epsilon\mathcal{H}_\theta\right)&u(x,t)=0,&(\epsilon>0)
\end{align}
for the Tate elliptic curve $E_q(K)$ is studied. The goal is to answer the corresponding Cauchy Problem, verify that the heat operator $\epsilon\mathcal{H}_\theta$ generates a Markov process on $E_q(K)$ and write down the fundamental solution of (\ref{heatEquation}).
\newline

The approach, as suggested by the anonymous referee, will be to do this in a much more general setting in the following subsection, and then to specialise that result to the case of the heat operator $\epsilon\mathcal{H}_\theta$ on $E_q(K)$. Subsection \ref{section_Berkovich} then will study a corresponding diffusion on the Berkovich-analytification of $E_q$.

%%%%%%%%%%%%%%%%%%%
\subsection{Heat equation on $K$-analytic manifolds}

Let $X$ be an $n$-dimensional $K$-analytic manifold, and let $\omega$ be a regular $n$-form on $X$. The gauge measure $\absolute{\omega}$ from
\cite[Ch.\ II.2.2]{WeilAAG}
can be defined in any dimension $n\in\mathds{N}$ outside of the vanishing set $V(\omega)\subset X$. Again, using \cite[Lem.\ 3.1]{Yasuda2017}, it can be seen that $V(\omega)$ has measure  zero, so $\absolute{\omega}$ again has a natural extension to all of $X$. The idea to follow here is to study operators of the form
\begin{align}\label{advection}
\mathcal{J}f(x)=\int_X
\left\{j(x,y)f(y)-j(y,x)f(x)\right\}\absolute{\omega(y)}
\end{align}
as considered in \cite{ZunigaEigen,ZunigaRugged}. In order to do this, write
\[
X\times X=E_+\cup E_-
\]
with
\[
E_+\cap E_-=\Delta_X:=\mathset{(x,x)\mid x\in X}
\]
and a map
\[
\iota\colon E_+\to E_-,\;(x,y)\mapsto (y,x)
\]
which is bijective. Let 
\[
j_+\colon E_+\to\mathds{R}_{\ge0},\quad
j_-\colon E_-\to\mathds{R}_{\ge0}
\]
be maps with
\[
j_+|_{\Delta_X}=j_-|_{\Delta_X}
\]
satisfying the following hypothesis:
\begin{align}\label{hypothesis}
j_-(x,y)\le j_+(y,x)\quad\text{for}\quad (x,y)\in E_-
\end{align}
and set
\begin{align*}
\mathcal{J}_-f(x)
&=\int_{\mathset{y\colon (x,y)\in E_-}}
\left\{j_-(x,y)f(y)-j_+(y,x)f(x)\right\}\absolute{\omega(y)}
\\
\mathcal{J}_+f(x)
&=\int_{\mathset{y\colon (x,y)\in E_+}}
\left\{
j_-(y,x)f(y)-j_+(x,y)f(x)
\right\}\absolute{\omega(y)}
\end{align*}
in order to obtain the operator
\[
\mathcal{J}=\mathcal{J}_-+\mathcal{J}_+
\]
which acts on  the space $C_0(X,\mathds{R})$ of continuous functions vanishing at infinity via (\ref{advection}). 

\begin{Lemma}\label{maxPos}
The operator $\mathcal{J}$
satisfies the positive maximum principle.
\end{Lemma}

\begin{proof}
Let $f\in C_0(X,\mathds{R})$, and let $x_0\in X$ be a place where $f$ takes its maximum, assumed positive.
Then
\begin{align*}
\mathcal{J}f(x_0)
&\le\int_X \left\{ j(x_0,y)
-j(y,x_0)
\right\}\absolute{\omega(y)}f(x_0)\stackrel{(*)}{\le} 0\cdot f(x_0)\le 0
\end{align*}
where $(*)$ holds true because of (\ref{hypothesis}).
\end{proof}

In order to be able to use the other conditions of the Hille-Yosida Theorem, observe that since $X$ need not be compact,
and $\mathcal{J}$ may also be unbounded, a different approach than in the proof of \cite[Thm.\ 3.1]{ZunigaRugged}. The approach taken here is similar to that of the proof of \cite[Lem.\ 5.1]{brad_SchottkyDiffusion}.
\newline

First, denote $p_1,p_2\colon X\times X\to X$ the projection onto the first, and second coordinate, respectively. Let
$\mathcal{P}(j)\subset X\times X$
be the set of poles of the kernel function $j(x,y)$, and let 
\[
P_z=p_2\left(p_1^{-1}(z)\cap \mathcal{P}(j)\right)
\]
for $z\in X$. Now, fix an atlas $\mathcal{U}$ of $X$, and set
\[
U_k(z):=\mathset{y\in X\mid\text{locally w.r.t.\ $\mathcal{U}\colon d(y,P_z)\le\absolute{\pi}^k$}}
\]
for $k\in\mathds{N}$, where $d$ is the distance in $U\in\mathcal{U}$, viewed as an analytic subdomain of $K^n$.
Now, define
\[
X_k(z)=X\setminus U_k(z)
\]
and
\[
\deg_k(z)=\int_{X_k(z)}j(y,z)\absolute{\omega(y)}
\]
for $k\in\mathds{N}$ and $z\in X$.

\begin{Assumption}\label{assumptionProper}
The manifold $X$ is assumed to be an open submanifold of a compact $K$-analytic manifold $\bar{X}$, and that $X\setminus\bar{X}$ is  a zero set w.r.t.\ the measure $\absolute{\omega}$.
The measure $\absolute{\omega}$ itself is assumed integrable on any chart $U\to K^n$, i.e.\ the differential form $\omega$ being on $U$ of the form
\[
\omega|_U=f_U\,dx_1\wedge\dots\wedge dx_n
\]
with $f_U$ a $K$-analytic function on $U$, satisfies
\[
\int_U\absolute{f_U}\absolute{dx_1\wedge\dots\wedge dx_n}<\infty
\]
where
$\absolute{dx_1\wedge\dots\wedge dx_n}$ is the normalised Haar measure on $K^n$.
\end{Assumption}

\begin{Assumption}\label{assumptionPoles}
It is assumed that $\mathcal{P}(j)$ is nowhere dense in $X\times X$, and that
\[
\mu_{X^2}(\mathcal{P}(j))=\int_{\mathcal{P}(j)}\absolute{\omega}\wedge\absolute{\omega}=0
\]
holds true.
\end{Assumption}

\begin{Assumption}\label{assumptionDegree}
It is assumed that
\[
\deg_k(z)<\infty
\]
for all $z\in X$ and $k\in\mathds{N}$. 
\end{Assumption}

Although the following result is not needed in the full generality of this subsection, it is nevertheless of independent interest.

\begin{thm}\label{AssumptionsFeller}
Under Assumptions \ref{assumptionProper}, \ref{assumptionPoles}, \ref{assumptionDegree}, and Hypothesis (\ref{hypothesis}), the linear operator $\mathcal{J}$ generates a Feller semigroup $e^{t\mathcal{J}}$ ($t\ge0$) on $C_0(X,\mathds{R})$.
\end{thm}

\begin{proof}
This is shown by checking the requirements for the Hille-Yosida-Ray Theorem \cite[Ch.\ 4, Lem.\ 2.1]{EK1986}, cf.\ 1.-3.\ below.
\newline

1. The domain of $\mathcal{J}$ is dense in $C_0(X,\mathds{R})$.
This follows from Assumption \ref{assumptionProper}.
\newline

2. The operator $\mathcal{J}$ satisfies the positive maximum principle. This was proven in Lemma \ref{maxPos}.
\newline

3. $\rank(\eta I-\mathcal{J})$ is dense in $C_0(X,\mathds{R})$ for some $\eta>0$. Since $\mathcal{J}$ could be unbounded, a proof as in \cite[Thm.\ 3.1]{ZunigaRugged} does not cover all cases here. Therefore, what follows is modelled after the proof of \cite[Lem.\ 5.1]{brad_SchottkyDiffusion}. The task is to find a solution of the equation
\begin{align}\label{taskEq}
(\eta I-\mathcal{J})u=h
\end{align}
for some $\eta>0$, and $h$ in some dense subspace of $C_0(X,\mathds{R})$. The equation can be formally rewritten as
\begin{align}\label{rewriteEq}
u(z)-\frac{\int_X j(z,y)u(y)\absolute{\omega(y)}}{\eta-\deg(z)}=\frac{h(z)}{\eta-\deg(z)}
\end{align}
with
\[
\deg(z)=\int_X j(y,z)\absolute{\omega(y)}
\]
which possibly does not converge, as $\mathcal{J}$ might be unbounded. For this reason, study the operator
\[
T_k u(z)=\frac{\int_{X_k(z)} j(y,z)u(y)\absolute{\omega(y)}}{\eta-\deg_k(z)}
\]
for $k>>0$. Now,
\[
\absolute{T_k u(z)}\le
\frac{\deg_k(z)}{\absolute{\eta-\deg_k(z)}}\norm{u}_\infty<\infty
\]
by Assumption \ref{assumptionDegree}.
This implies that
\[
\norm{T_k}\le\frac{1}{1-\eta/\deg_k(z)}<1
\]
for $k>>0$. Hence, $I-T_k$ has a bounded inverse as an operator on $C_0(X,\mathds{R})$. Consequently, the range of $I-T_k$ is dense in $C_0(X,\mathds{R})$ for $k>>0$.
Now, let $h\in\mathcal{D}(X,\mathds{R})$, the space of locally constant real-valued functions with compact support, and let $u_k,u_\ell\in C_0(X,\mathds{R})$ be solutions of
\begin{align*}
(I-T_k)u_k&=\frac{h}{\eta-\deg_k}
\\
(I-T_\ell)u_\ell&=\frac{h}{\eta-\deg_\ell}
\end{align*}
for $k,\ell>>0$. Then
\begin{align}\label{CauchySeq}
u_k-u_\ell=
\frac{(I-T_\ell)(\eta-\deg_\ell)-(I-T_k)(\eta-\deg_k)}{(I-T_k)(I-T_\ell)(\eta-\deg_k)(\eta-\deg_\ell)}\,h
\end{align}
shows that $(u_k)$ is a Cauchy sequence w.r.t.\ $\norm{\cdot}_\infty$. Namely, first
\begin{align}\label{operatorNorm}
\norm{T_k}=\sup\limits_{z\in X}
\absolute{\frac{\deg_k(z)}{\eta-\deg_k(z)}}
=\sup\limits_{z\in X}\frac{1}{1-\eta/\deg_k(z)}
\end{align}
is strictly increasing to $1$ for $k\to\infty$. Hence, $T_k$ converges to a bounded linear operator $T$ on $C_0(X,\mathds{R})$. Secondly, the numerator of the right hand side of (\ref{CauchySeq}) is
\[
\eta(T_k-T_\ell)+(\deg_k-\deg_\ell)+(T_\ell\deg_\ell-T_k\deg_k)
\]
whose first and second terms in norm become arbitrarily small as $\ell\ge k\to\infty$. The third term is
\[
T_\ell \deg_\ell-T_k\deg_k
=(T_\ell\deg_\ell-T_k\deg_\ell)+(T_k\deg_\ell-T_k\deg_k)
\]
both of whose summands converge to $0$ as $\ell\ge k\to\infty$. It follows that $u_k$ convertes to some $u\in C_0(X,\mathds{R})$ which is seen to be a solution of (\ref{taskEq}) as follows:
Namely, $(\eta+\deg_k)T_k$ converges to $(\eta+\deg)T$ for $k\to\infty$, where the limit operator coincides with the operator
\[
Au(z)=\int_X j(z,y)u(y)\absolute{\omega(y)}
\]
which shows that the operator
\[
T=\frac{A}{\eta-\deg}
\]
appearing in (\ref{rewriteEq}) is bounded. Now, $u_k$ is a solution of
\[
(\eta I-\mathcal{J}_k)u_k=h
\]
with
\[
\mathcal{J}_k=(\eta-\deg_k)T_k-\deg_k
\]
which converges to $\mathcal{J}$ for $k\to\infty$. As $u_k$ converges to $u$, it follows that
\[
(\eta I-\mathcal{J})u=(\eta I-\mathcal{J}_k)u+(\mathcal{J}_k-\mathcal{J})u
\]
where 
\[
(\eta I-\mathcal{J}_k)u
=(\eta I-\mathcal{J}_k)u_k+\mathcal{J}_k(u_k-u)
=h+\mathcal{J}_k(u_k-u)
\]
converges to $h$ for $k\to\infty$, and
\[
(\mathcal{J}_k-\mathcal{J})u\to 0
\]
for $k\to\infty$. Hence, $u$ is a solution of (\ref{taskEq}). This proves that $\rank(\eta I-\mathcal{J})$ contains $\mathcal{D}(X,\mathds{R})$ which is dense in $C_0(X,\mathds{R})$.

\smallskip
Since limit operator of $T_k$does not depend on the choice of an atlas $\mathcal{U}$, cf.\ (\ref{operatorNorm}), it follows that the existence of the solution $u$ of (\ref{taskEq}) does also not depend on the choice of an atlas.
This now proves the assertion.
\end{proof}

\begin{thm}\label{MarkovPropertyMf}
There exists a probability measure $p_t(x,\cdot)$ with $t\ge0$, $x\in X$, on the Borel $\sigma$-algebra of $X$ such that the Cauchy problem 
\begin{align*}
u(\cdot,t)&\in C^1\left([0,\tau],C_0(X,\mathds{R})\right)
\\
\frac{\partial}{\partial t}u(x,t)&
=\int_X \left(j(x,y)u(y,t)-j(y,x)u(x,t)\right)\absolute{\omega(t)},&t\in[0,\tau],\;x\in X
\\
u(x,0)&=u_0(x)\in C_0(X,\mathds{R})
\end{align*}
has a unique solution of the form
\[
h(t,x)=\int_{X}h_0(y)\,p_t(x,\absolute{\omega(y)})
\]
In addition, $p_t(x,\cdot)$ is the transition function of a strong Markov process whose paths are right continuous and have no discontinuities other than jumps.
\end{thm}

This is Theorem \ref{secondTheorem} of the Introduction.

\begin{proof}
The proof will  be the same as in the case of \cite[Thm.\ 4.2]{ZunigaNetworks}, and is given here for the convenience of the reader. 

\smallskip
Using the correspondence between Feller semigroups and transition functions, one sees from Theorem  \ref{AssumptionsFeller} that there is a uniformly stochastically continuous $C_0$-transition function $p_t(x,\absolute{\omega})$ satisfying condition (L) of \cite[Thm.\ 2.10]{Taira2009} such that
\[
\exp\left(t\mathcal{J}\right)h_0(x)
=\int_{X}h_0(y)\,p_t(x,\absolute{\omega(y)})
\]
for $h_0\in C_0(X,\mathds{R})$, cf.\ e.g.\ \cite[Thm.\ 2.15]{Taira2009}. Now, by using the correspondence between transition functions and Markov processes, there exists a strong Markov process whose paths are right continuous and have no discontinuities
other than jumps, see e.g.\ \cite[Thm.\ 2.12]{Taira2009}.
\end{proof}

Now, the following task will be addressed:

\begin{task}[Cauchy Problem]\label{CauchyProblem}
Find $h(t,x)\in C^1\left((0,\infty),E_q(K)\right)$ such that
\begin{align*}
\left(\frac{\partial}{\partial t}-\epsilon\mathcal{H}_\theta\right)h(t,x)&=0
\\
h(0,x)&=h_0(x)
\end{align*}
for $t\ge0$, $h_0\in C(E_q(K))$.
\end{task}

\begin{Corollary}\label{MarkovProperty}
There exists a probability measure $p_t(x,\cdot)$ with $t\ge0$, $x\in E_q(K))$, on the Borel $\sigma$-algebra of $E_q(K)$ such that the Cauchy problem (Task \ref{CauchyProblem})
has a unique solution of the form
\[
h(t,x)=\int_{E_q(K)}h_0(y)\,p_t(x,\absolute{\omega(y)})
\]
In addition, $p_t(x,\cdot)$ is the transition function of a strong Markov process whose paths are right continuous and have no discontinuities other than jumps.
\end{Corollary}

\begin{proof}
Hypothesis (\ref{hypothesis}) is naturally satisfied by the kernel function. As Assumptions \ref{assumptionProper}, \ref{assumptionPoles}, and \ref{assumptionDegree} are clearly satisfied, Theorem \ref{AssumptionsFeller} can be applied, and thus Theorem \ref{MarkovPropertyMf} yields the assertions.
\end{proof}

As the corresponding semigroup is Feller, it describes a $\pi$-adic heat equation on $E_q(K)$. Consequently, there is a corresponding $\pi$-adic diffusion process in $E_q(K)$ attached to the heat equation (\ref{heatEquation}).

\begin{remark}
As it was  pointed out by the anonymous referee, it is possible to see Corollary \ref{MarkovProperty} as a special case of 
 of \cite[Thm.\ 1]{ZunigaEigen}, which itself finds an extension  to general types of $K$-analytic manifolds in Theorem \ref{MarkovPropertyMf} here.
\end{remark}

\subsection{Diffusion on the Berkovich-Analytification}\label{section_Berkovich}

In order to extend the diffusion theory to the Berkovich-analytification $E_q^{an}$ of the Tate curve $E_q$, it is necessary to solve the so-called Monge-Amp\`ere equation
\begin{align}\label{MongeAmpereEquation}
\mu=\lambda \, c_1(\mathcal{L},\norm{\cdot})
\end{align}
where $\mu$ is a positive Radon measure on $E_q^{an}$. What this means is to find $\lambda>0$ and a line bundle $L$ on $E_q^{an}$ such that  
 the \emph{Chambert-Loir measure} $c_1(\mathcal{L},\norm{\cdot})$  on $E_q^{an}$ coming from  a subharmonic metric $\norm{\cdot}$ on the line bundle $L$ satisfies equation (\ref{MongeAmpereEquation}). 
Since $E_q^{an}$ is a curve, this equation %(\ref{MongeAmpereEquation}) 
has been found to always have a solution by Thuillier in his dissertation \cite[Cor.\ 3.4.18]{Thuillier2005}. The solvability of the Monge-Amp\`ere equation in a much more general non-archimedean setting is settled in e.g.\
\cite{FGK2022} in the mixed characterstic case.
\newline

Recall that a \emph{harmonic function} on a weighted metrised graph $G$ 
is a pieceweise affine function  $f\colon G\to\mathds{R}$ such that for each point
\[
\lambda_x(f):=\sum\limits_{t\in T_x(G)}
w_x(t)\lambda_{x,t}=0
\]
where $T_x(G)$ is the set of all tangent directions in each point of $G$, $w_x(t)$ the weight function in the point $x$ along direction $t$, and $\lambda_{x,t}$ the linear part in the affine piece of $f$ in the point $x$ along $t$:
\[
f=f(x)+\lambda_{x,t}(f)t
\]
is the affine representation of $f$ 
for this particular pair $(x,t)$.
The corresponding map
\[
\ddc^c\colon A^0(E_q^{an})\to {A}^1(E_q^{an}),\;
f\mapsto\sum\limits_{x\in G}\lambda_{x}(f)\,\delta_x
\]
is called \emph{Laplacian}, and a the kernel of $\ddc^c$ consists of the harmonic functions in the space $A^0(E_q^{an})$ of global sections of piece-wise affine functions on $E_q^{an}$, whereas $A^1(E_q^{an})$ are  the global sections of the sheaf  of locally finite signed measures on $E^{an}_q$.
\newline

A 
subharmonic function on an open $U\subseteq E_q^{an}$ is a function
\[
u\colon U\to\mathds{R}\cup\mathset{-\infty}
\]
which is upper half-continuous, not identically $-\infty$ on any connected component of $U$, and for any strictly $K$-affinoid domain $Y$ of $E_q^{an}$ any harmonic function $h$ on $Y$, it holds true that
\[
u|_{\partial Y}\le h|_{\partial Y}\quad\Rightarrow\quad u|_Y\le h
\]
A subharmonic metric on an invertible sheaf $\mathcal{L}$ is a family of  functions
\[
\norm{\cdot}\colon 
\Gamma(U,\mathcal{L}^\times)\to\mathds{R}
\]
with $U\subset E_q^{an}$ running through the open subsets, 
such that $-\log\norm{s}$ is a subharmonic function on $U$, and for all open $U'\subset U\subset E_q^{an}$ and sections
$s'\in \Gamma(U',\mathcal{L}^\times)$, $s\in \Gamma(U,\mathcal{L}^\times)$, it holds true that
$s|_{U'}=fs'$ for some $f\in \Gamma(U',\mathcal{O}_{E_q^{an}}^\times)$
such that
\[
-\log\norm{s}|_{U'}=
-\log\norm{s'}-\log\absolute{f}
\]
Notice that the function $\log\absolute{f}$ is subharmonic \cite[Prop.\ 3.1.6]{Thuillier2005}.
\newline

If $(\mathcal{L},\norm{\cdot})$ is a metrised line bundle on $E_q^{an}$, then
\[
\ddc^c\log\norm{s}
\]
for a section $s\in \Gamma(U,\mathcal{L}^\times)$ is in general only a distribution on $A^1(E_q^{an})$. Such is called a \emph{current} of degree $1$ on $E^{an}_q$.
This current associated with $(\mathcal{L},\norm{\cdot})$ is denoted as
\[
c_1(\mathcal{L},\norm{\cdot})
\]
and is called its \emph{curvature form}.
It is positive, iff $\norm{\cdot}$ is a subharmonic metric on $\mathcal{L}$. And equation (\ref{MongeAmpereEquation}) asks for the existence of a curvature form representing a given positive Radon measure on $E_q^{an}$.
\newline

In order to relate the theory of metrised line bundles to the operator $\mathcal{H}_\theta$ on $E_q(K)$, observe that
the measure
\[
\mu_x=H_\theta(x,\cdot)\absolute{\omega(\cdot)}
\]
is a  positve Radon measure on $E_q(K)$. Hence, its pushforward
\[
\mu_{\sigma(x)}:=\sigma_*\mu_x
\]
is a  Radon measure on the skeleton $I(E_q^{an})$, and supported on $\sigma(E_q(K))$.
Define also
\[
\deg_\theta(\sigma(x)):=\int_{I(E_q^{an})}\mu_{\sigma(x)}
\]
for $x\in E_q(K)$,
and solve the Monge-Amp\`ere equation
\[
\delta_z=c_1(\mathcal{L}
%\mathcal{O}_{E_q}
,\norm{\cdot}_z)
\]
on the curve $E_q^{an}$,
which is possible according to Thuillier with an ample line bundle $\mathcal{L}$ on $E_q^{an}$.
Notice that this is in fact the Calabi-Yau problem, where an ample line bundle $\mathcal{L}$ on a projective $K$-variety $X$ is given, and the question is whether  for any given positive Radon measure $\nu$ on $X^{an}$ with $\nu(X^{an})=\deg_{\mathcal{L}}(X)$, there exists a continuous semipositive
metric $\norm{\cdot}$ on $\mathcal{L}$ such that
\[
\nu=c_1(\mathcal{L},\norm{\cdot})^n
\]
with $\dim(X^{an})=n$.
This has  been 
recently solved in several 
important instances, e.g.\  if $n=1$ and $X$ is a smooth (analytic) curve \cite{Thuillier2005},   in case the residue characteristic is zero and $X$ is smooth \cite{BFJ2016}, 
and in the mixed characteristic case for smooth $X$ \cite{FGK2022}.
\newline

A \emph{model} of the variety $X$ is a normal scheme $\mathcal{X}\to\Spec O_K$ which is flat and whose generic fibre is isomorphic to $X$. Given an ample line bundle $\mathcal{L}$ on $X$, a \emph{model metric} $\norm{\cdot}_{\mathscr{L}}$ on $\mathcal{L}$ is defined by an extension $\mathscr{L}\in \Pic(\mathcal{X})_{\mathds{Q}}$ of $\mathcal{L}$ to some model $\mathcal{X}$ of $X$. 
A model metric is \emph{semipositive}, if 
the line bundle $\mathcal{L}$ is nef, i.e.\ if  
\[
\deg_{\mathscr{L}}(C)\ge0
\]
for all proper curves $C$ on the special fibre $\mathcal{X}_s$ of $\mathcal{X}$.
A \emph{semipositive continuous} metric $\norm{\cdot}$ on $\mathcal{L}$ is a uniform limit of semipositive metrics on $\mathcal{L}$. This gives rise to the \emph{Chambert-Loir measure} $c_1(\mathcal{L},\norm{\cdot})^n$, a positive Radon measure on $X^{an}$ of mass $c_1(\mathcal{L})^n$, cf.\  \cite{CL2006}.
So, the Calabi-Yau problem is to find a suitable continuous semipositive metric on $\mathcal{L}$ such that the corresponding Chambert-Loir measure coincides with the given positive Radon measure.
\newline

In any case, 
 obtain now a measure
\[
\nu_{\sigma(x)}=\mu_{\sigma(x)}-\deg_\theta(\sigma(x))\delta_{\sigma(x)}
\]
on $E_q^{an}$
for $x\in E_q(K)$ which is supported on the skeleton $\sigma(E_q(K))\subset I(E_q^{an})$.
This yields an operator:
\[
\mathcal{H}_{\theta,\sigma}\phi(z)
=\int_{I(E_q^{an})}
\phi\,d\nu_z
\]
with  $\phi\in\mathcal{D}(E_q^{an})$,
where $\nu_z$ is the measure defined as
\[
\nu_z=\begin{cases}
\nu_{\sigma(x)},&z=\sigma(x)\;\text{for some}\;x\in E_q(K)
\\
0,&\text{otherwise}
\end{cases}
\]
for $z\in E_q^{an}$.

\begin{definition}
A function $\phi\colon E^{an}_q\to\mathds{C}$ is \emph{radial}, if
\[
\phi(x)=\phi(\sigma(x))
\]
for all $x\in E_q^{an}$.
The space of all radial test functions is denoted as $\mathcal{D}(E_q^{an})_\sigma$.
\end{definition}

Helpful now is the averaging operator:
\[
\mathcal{D}(E_q^{an})\to
\mathcal{D}(E_q^{an})_\sigma,\;
\phi
\mapsto\average(\phi)
\]
with $\average(\phi)$ defined as
\[
x\mapsto
\left(\int_{U(x)}\absolute{\sigma_*\omega}\right)^{-1}
\int_{U(x)}\phi(y)\absolute{\sigma_*\omega(y)}
\]
where 
\[
U(x):=\sigma^{-1}(\sigma(x))
\]
is an open neighbourhood of $x\in E_q^{an}$.
Clearly, functions of the form $\average(\phi)$ are precisely the radial functions. Using the obviously defined pairing
$\langle\cdot,\cdot\rangle_{\sigma_*\omega}$ on $L^2(E_q^{an},\absolute{\sigma_*\omega})$, obtain
the orthogonal decomposition
\[
L^2(E_q^{an},\absolute{\sigma_*\omega})=L^2(E_q^{an})_\sigma\oplus L^2(E_q^{an})_0
\]
where $L^2(E_q^{an})_\sigma$ is spanned by the radial functions, and whose orthogonal complement is $L^2(E_q^{an})_\sigma^\perp=L^2(E_q^{an})_0$. 
\newline

Let 
\[
U(\sigma(E_q(K))) := \sigma^{-1}(\sigma(E_q^{an}(K)))
\]
The  space
$L^2(U(E_q(K)))\subset L^2(E_q^{an},\absolute{\sigma_*\omega})$
is a closed subspace and thus has the same property
\begin{align}\label{orthogonalDecompositionU}
L^2(U(E_q(K)))=L^2(U(E_q(K)))_\sigma\oplus L^2(U(E_q(K)))_0
\end{align}
and the part $L^2(U(E_q(K)))_0$ is spanned
by the wavelets on $E_q(K)$. The reason is that that space coincides with the space of $L^2$-functions on $E_q^{an}(K)$ supported outside the skeleton averaging to zero.

\begin{Lemma}\label{skeletonOperator}
Assume that $\phi\in \mathcal{D}(E_q^{an})$ is  a radial function. 
It holds true that 
\[
\mathcal{H}_{\theta,\sigma}\phi(z)=\begin{cases}
\mathcal{H}_\theta\phi(x),&\exists\,x\in E_q(K)\;\text{with}\;\sigma(x)=\sigma(z)
\\
0,&\text{otherwise}
\end{cases}
\]
where $z\in E^{an}_q$.
\end{Lemma}

\begin{proof}
This is clear from the fact that $\nu_{\sigma(x)}$ is supported on $\sigma(E_q(K))$.
\end{proof}

\begin{Lemma}\label{waveletOperator}
Assume that $\phi\in L^2(U(E_q(K)))_0$. Then
\[
\mathcal{H}_{\theta,\sigma}\phi=\mathcal{H}_\theta\phi
\]
holds true.
\end{Lemma}

\begin{proof}
This follows immediately from the discussion before Lemma \ref{skeletonOperator}.
\end{proof}

\begin{Corollary}\label{MarkovProcessBerkovichTateCurve}
The operator $\mathcal{H}_{\theta,\sigma}$ is a self-adjoint bounded linear operator on $L^2(E_q^{an}(K))$.
Its spectrum coincides with that of $\mathcal{H}_\theta$ acting on $L^2(E_q^{an},\absolute{\omega})$,  and there exists a probability measure $p_t(x,\cdot)$ with $t\ge0$, $x\in E_q^{an}(K)$, on the Borel $\sigma$-algebra of $E_q^{an}(K)$ such that the Cauchy problem for the heat equation with $\mathcal{H}_{\theta,\sigma}$ has a unique solution of the form 
\[
h(t,x)=\int_{E_q^{an}(K)}h_0(y)\absolute{\sigma_*\omega(y)}
\]
where $h_0\in C(E_q^{an}(K))$ is an initial condition for the heat equation. Furthermore, $p_t(x,\cdot)$ is the transition function of a Markov process whose paths
are right continuous and have no discontinuities other than jumps.
\end{Corollary}

\begin{proof}
This follows from the orthogonal decomposition
(\ref{orthogonalDecompositionU}), the isomorphism 
\[
L^2(U(E_q(K)))_0\cong L^2(E_q(K))_0
\]
observed before Lemma \ref{skeletonOperator},
the isomorphism
\[
L^2(U(E_q(K)))_\sigma=L^2(E_q(K))_\sigma
\]
between the radial parts, and Lemmas \ref{skeletonOperator} and \ref{waveletOperator},  and using Theorem \ref{SpectrumHeatOperator} and Corollary \ref{MarkovProperty}.
Notice that the Feller semigroup property of the operator $\mathcal{H}_{\theta,\sigma}$ restricted to $C(U(E_q(K)))$ requires \cite[Prop.\ 15]{brad_heatMumf}.
\end{proof}

Apart from having a diffusion operator on the Berkovich-analytifi\-ca\-tion of a Tate curve, the added value of Corollary \ref{MarkovProcessBerkovichTateCurve} is that
the operator $\mathcal{H}_{\theta,\sigma}$ is controlled also by the Chambert-Loir measure $c_1(\mathcal{L}
%\mathcal{O}_{E_q}
,\norm{\cdot}_{\sigma(x)})$ equaling the Dirac measure on the point $\sigma(x)$ on the skeleton of the Tate curve. And also the push-forward of the positive Radon measure 
\[
\mu_x=H_\theta(x,\cdot)\absolute{\omega(\cdot)}
\]
can be written as a Chambert-Loir measure
\[
\mu_{\sigma(x)}=c_1(\mathcal{L}
%\mathcal{O}_{E_q}
,\norm{\cdot}_{\theta,x})
\]
since we are dealing with the case of a curve.

\begin{prop}\label{heatMeasureCL}
The measure $\nu_{\sigma(x)}$ can be written as
\[
\nu_{\sigma(x)}=\sum\limits_{z\in\sigma(E^{an}_q(K))}
\alpha_z\, c_1\left(\mathcal{L}%\mathcal{O}_{E_q^{an}}
,\norm{\cdot}_z\right)-\deg_\theta(\sigma(x))\,c_1\left(\mathcal{L}
%\mathcal{O}_{E_q^{an}}
,\norm{\cdot}_{\sigma(x)}\right)
\]
with $\alpha_z>0$, and
$\norm{\cdot}_y=e^{-g_y}
$
with $g_y$ a continuous subharmonic function on $E_q^{an}$
such that
\[
c_1\left(\mathcal{L}
%\mathcal{O}_{E_q^{an}}
,\norm{\cdot}_y\right)=\delta_y
\]
for $y\in \sigma(E^{an}_q(K))\sqcup\mathset{\sigma(x)}
$
and $x\in E_q(K)$. 
\end{prop}

\begin{proof}
Clearly, $\nu_{\sigma(x)}$ is a linear combination of Dirac measures:
\[
\nu_{\sigma(x)}
=\sum\limits_{z\in \sigma(E_q^{an}(K))}
\alpha_z\,\delta_z-\deg_\theta(\sigma(x))\,\delta_{\sigma(x)}
\]
supported in $\sigma(E_q(K))$. 
Using the isomorphism 
\[
L^2(E(K),\absolute{\omega})\cong L^2(U(E_q(K)))
\]
obtained in the proof of Corollary \ref{MarkovProcessBerkovichTateCurve}, observe that
the elements $\alpha_z$ are the corresponding entries in the helpful matrix $L$ of (\ref{helpfulMatrix}).
Using \cite[Lem.\ 3.4.14]{Thuillier2005}, find that the Dirac measures are now seen to be of the form
\[
\delta_y=\ddc^c(g_y)
\]
with $g_y$ subharmonic on $E_q^{an}(K)$ for $z\in \sigma(E_q^{an}(K))\sqcup\mathset{\sigma(x)}$. Then $\norm{\cdot}_y=e^{-g}$ are the desired corresponding metrics on $\mathcal{L}$
%\mathcal{O}_{E_q}$ 
proving the assertion.
\end{proof}

\begin{Corollary}\label{ChernOperator}
The heat operator $\mathcal{H}_{\theta,\sigma}$ on $E_q^{an}(K)$ is obtained as an integral operator of the form
\[
\mathcal{H}_{\theta,\sigma}\psi(x)=\int_{E_q^{an}}\psi \,c_1\left(\mathcal{L}
%\mathcal{O}_{E_q^{an}}
,\norm{\cdot}_{\theta,x}\right)
\]
with metric
$
\norm{\cdot}_{\theta,x}=e^{-g_{\theta,x}}
$
and
\[
g_{\theta,x}=\sum\limits_{z\in \sigma(E_q^{an}(K))}\alpha_z g_z
\]
where $\alpha_z$ and $g_z$ are as in Proposition \ref{heatMeasureCL}, and $x\in\sigma(E_q(K))$.
\end{Corollary}

This is Theorem \ref{thirdTheorem}
of the Introduction.

\begin{proof}
This is an immediate consequence of Proposition \ref{heatMeasureCL}.
\end{proof}
%%%%%%%%%%%%%%%%
\section{On Hearing the Shape of a Tate Curve}

In \cite[Cor.\ VI.1]{BL_shapes_p}, it was shown how to reconstruct a reduction graph of a Mumford curve from the spectrum of a $p$-adic diffusion operator having an infinitely-valued spectrum. The spectrum of the operator $\mathcal{H}_\theta$ has only  finitely many values, and can also recover some information about the  $K$-rational points of a Tate curve $E_q$ themselves. This is shown in the following theorem:

\begin{thm}\label{applicationTheorem}
Let $E_q$ be a Tate elliptic curve defined over a %sufficiently ramified 
local field 
$K$. Assume that all about $K$ is known. Then the spectrum of the Laplacian operator $\mathcal{H}_\theta$
on $L^2(E_q(K),\absolute{\omega})$,
or of the Laplacian operator $\mathcal{H}_{\theta,\sigma}$ on $L^2(U(E_q(K)))$,
respecively, can detect the following:
\begin{enumerate}
\item   the presence of a $2$-torsion point in $E_q(K)$.
\item  the parity of $v(q)$.
\item the presence of $\pi^\ell\in E_q(K)$ which is a third root of $\pi^{-k}$, where $k,\ell$ are solutions for the congruence
\[
k+3\ell\equiv 0\mod v(q)
\]
in $\mathds{Z}/v(q)\mathds{Z}$.
\end{enumerate}
\end{thm}

\begin{proof}
From Lemmas \ref{skeletonOperator} and \ref{waveletOperator}, it follows that the spectra of both operators coincide.
 The assertions now follows from the explicit form of the negative degree eigenvalues given in Corollary \ref{degreeFunction}.
\end{proof}

%%%%%%%%%%%%%%%%%%%%%%%%%%%%%%%%%%%%%%%
\section*{Acknowledgements}
Yassine El Maazouz, Stefan K\"uhnlein, \'Angel Mor\'an Ledezma and David Weisbart are warmly thanked for fruitful discussions. 
Frank Herrlich is thanked for guiding the author towards this beautiful area of mathematics when he was his student.
The anonymous referee is thanked for helpful and clarifying suggestions towards substantially improving the paper.
This work is partially supported by the Deutsche Forschungsgemeinschaft under project number 469999674.

%%%%%%%%%%%%%%%%
\bibliographystyle{plain}
\bibliography{biblio}

\end{document}